\documentclass[oneside,12pt]{amsart}
\usepackage{eurosym}
\usepackage{amsfonts} 
\usepackage{amsmath,amssymb,amsthm,xcolor}
\usepackage{latexsym}
\usepackage{float}
\usepackage{placeins}
\usepackage{bm}
\usepackage[utf8]{inputenc}
\usepackage[british]{babel}
\usepackage{a4wide}
\usepackage{pdfsync}
\usepackage[colorlinks=true,linkcolor=blue,citecolor=blue]{hyperref}
\usepackage[all]{xy}
\usepackage{parskip}
\usepackage{multirow}
\usepackage{array}
\usepackage[toc,page]{appendix}
\usepackage{graphicx}
\usepackage{comment}

\theoremstyle{definition}
\newtheorem{lemma}{Lemma}

\newtheorem{theorem}[lemma]{Theorem}
\newtheorem*{theorem*}{Theorem}

\newtheorem{definition}[lemma]{Definition}

\newtheorem{remark}[lemma]{Remark}

\newcommand{\R}{\mathbb{R}}
\renewcommand{\H}[1]{\mathbb{H}^{#1}}

\newcommand{\grad}{\mathrm{grad}}

\renewcommand{\div}{\mathrm{div}}

\newcommand{\vv}{\mathbf{v}}
\newcommand{\ff}{\mathbf{f}}

\begin{document}

\title{Killing mean curvature solitons from Riemannian submersions}

\author{Diego Artacho}
\address{D.~Artacho: Department of Mathematics, Faculty of Natural Sciences, Imperial College London, 180 Queen's Gate London SW72AZ, UK}
\email{d.artacho21@imperial.ac.uk}

\author{Marie-Am\'{e}lie Lawn}
\address{M.-A.~Lawn: Department of Mathematics, Faculty of Natural Sciences, Imperial College London, 180 Queen's Gate London SW72AZ, UK}
\email{m.lawn@imperial.ac.uk}

\author{Miguel Ortega}
\address{M.~Ortega: Institute of Mathematics, Department of Geometry and Topology, Faculty of Sciences, Campus de Fuentenueva, 18079 Granada, Spain}
\email{miortega@ugr.es}

\begin{abstract}
   We present a new general construction of mean curvature flow solitons on manifolds admitting a nowhere-vanishing Killing vector field. Using Riemannian submersion techniques, we reduce the problem from a PDE to an ODE. As an application, we obtain new examples of rotators in hyperbolic space. 
\end{abstract}

\maketitle

\section{Introduction} \label{sec:intro}

Let $M^n$ and $N^{n+1}$ be smooth oriented manifolds. Let $g$ be a Riemannian metric on $N$, and $\nabla$ its Levi-Civita connection. A smooth one-parameter family of immersions $F \colon M \times [0, \varepsilon) \to N$ is said to \emph{evolve by mean curvature flow} if
\begin{equation}\label{eq:mcf}
\left( \frac{\partial F_t}{\partial t} \right)^{\perp} = \vec{H}_t \, ,
\end{equation}
where $F_t (x) := F(x,t)$, $(\cdot)^{\perp}$ indicates $g$-orthogonal projection to the normal bundle of $F_t$, and $\vec{H}_t$ is the mean curvature vector of $F_t$, given by 
\begin{equation*}
\vec{H}_t = \sum_{i=1}^{n} \left(\nabla_{e_i} e_i\right)^{\perp} 
\end{equation*}
if $(e_1 , \dots , e_n)$ is any local orthonormal frame along $F_t$. 

A particular kind of solution to~\eqref{eq:mcf} has attracted much interest in recent decades. If $\mathcal{G}$ is a smooth one-parameter family of isometries of $N$ with $\mathcal{G}_0 = \mathrm{Id}_N$, we say that a one-parameter family of immersions $F$ evolving by mean curvature flow is a \emph{$\mathcal{G}$-soliton} if
\begin{equation*}
F_t = \mathcal{G}_t \circ F_0 \, .
\end{equation*}
In the class of solitons, also called self-similar solutions, solving~\eqref{eq:mcf} amounts to finding a manifold $M$ and an immersion $F_0 \colon M \to N$ with mean curvature vector  
\begin{equation}\label{eq:mcf2}
    \vec{H} = K^{\perp} \, ,
\end{equation}
where $K$ is the Killing vector field determined by $\mathcal{G}$, i.e., 
\begin{equation} \label{eq:def_K}
    K(x) = \frac{d \mathcal{G}_t}{dt}\big\rvert_{t=0} (x) \, .
\end{equation}
This was proved in the seminal work~\cite{HS00}. If $\nu$ is a unit normal vector field to $F_0$, we can define the \emph{mean curvature with respect to $\nu$} to be $H := g \left( \vec{H} , \nu \right)$, and the vectorial equation~\eqref{eq:mcf2} becomes
\begin{equation}\label{eq:mcf3}
    H = g \left( K , \nu \right) \, .
\end{equation}
Those solutions which are invariant under a group of translations are the so-called \emph{translators}, which have been extensively studied in the literature -- see e.g.~\cite{CSS, LTW, MSHS}. In~\cite{ALR}, the concept of self-similar solutions is extended to a more general setting, allowing $\mathcal{G}$ to be a one-parameter family of \emph{conformal} maps. In particular, they study the example of a class of Riemannian warped products $P \times_h I$ as the ambient space where a conformal vector field naturally arises from the distinguished one-dimensional direction $I$. For the study of mean curvature flow in warped products, we also refer to~\cite{CMR, DS}, with~\cite{ALO} focusing on the Robertson-Walker case.

Graphical solutions to mean curvature flow were initially introduced by Ecker and Huisken \cite{H} for hypersurfaces in Euclidean space, and later extensively studied by numerous authors in various ambient spaces -- see e.g.~\cite{CSS, HIMW, LO, LM}. In all of them, they are graphs of solutions to a certain PDE. However, finding solutions to PDEs is generally a challenging task. A common strategy is to focus on solutions that are invariant under certain symmetries of the manifold, which can simplify the problem. We refer, for instance, to~\cite{CSS} for the study of solutions on $\mathbb{R}^n$ invariant by rotational symmetries. 

Our strategy is as follows. On a Riemannian manifold $(N,g)$, we first consider a Killing vector field $K$ that is nowhere zero, and the goal will be to construct solutions to mean curvature flow that are solitons with respect to the flow of $K$. In order to obtain solitons which can be described as solutions to a PDE, we need the orthogonal distribution to $K$ to be integrable, and all the leaves to be diffeomorphic to a fixed manifold $M$. In this way, $(N,g)$ inherits locally a warped product structure, and the soliton will be the graph of a smooth function over a subset of $M$. Secondly, we shall restrict our attention to solutions that are invariant under certain symmetries of $N$, usually coming from Lie group actions. The main idea of this paper is the use of \emph{special submersions} from $M$ to a one-dimensional base space that preserve key geometric properties, which in some cases allow one to reduce the PDE to an ODE.  

More precisely, we consider a manifold $(N,g)$ endowed with a warped product structure 
\[ 
(N,g) \cong M \times_{\varphi} I := \left( M \times I , g_M + \varphi dr^2 \right) \, , \] 
where $\left( M,g_M \right)$ is a Riemannian manifold, $I \subseteq \R$ is an interval, $\varphi \colon M \to (0,+\infty)$ is a smooth function, and $dr^2$ is the Euclidean metric on $I$. This suggests a natural one-parameter group of isometries of $N$, namely $\mathcal{G}_t (x,r) = (x , r+t)$. Moreover, there is a natural class of hypersurfaces of $N$, namely graphs of smooth functions $u \colon M \supset \Omega \to I$, where $\Omega$ is an open subset of $M$. In Section~\ref{sec:PDE}, we show that such a hypersurface is a $\mathcal{G}$-soliton if and only if $u$ satisfies a certain PDE - see Theorem~\ref{thm:pde}. This is a natural generalisation of the previously cited works. With this in mind, suppose  $\pi \colon \left( M,g_M \right) \to \left( J , \alpha(s) ds^2 \right)$ is a Riemannian submersion, where $J \subseteq \R$ is an interval, $\alpha \colon J \to (0,+\infty)$ is smooth, and $ds^2$ denotes the Euclidean metric on $J$. Furthermore, assume that both the function $u$ and $\varphi$ factor through $\pi$, i.e. $u = f \circ \pi$, and $\varphi=\widehat{\varphi}\circ \pi$ for some smooth functions $f,\, \widehat{\varphi}:J\rightarrow I$. If, moreover, $\pi$ has constant mean curvature fibres $h$, then the condition that the graph of $u$ defines a $\mathcal{G}$-soliton on the total space is equivalent to $f$ being a solution to an ODE on the base. This yields our main Theorem~\ref{thm:ode}, proved in Section~\ref{sec:ODE}: 
\begin{theorem*}[see Theorem~\ref{thm:ode}]
    In our previous setting, the graph of $u=f\circ \pi$ defines a $\mathcal{G}$-soliton on $N$ if and only if $f$ satisfies the ordinary differential equation:
\[ 
    f'' = \left(\alpha + \widehat{\varphi} (f')^2\right)
    \left(1 - \frac{\widehat{\varphi}' f'}{2 \widehat{\varphi} \alpha} - \frac{f' h}{\sqrt{\alpha}}\right)
    + \frac{f'}{2} \left( \log \left( \frac{\alpha}{\widehat{\varphi}} \right) \right)'.
\] 
\end{theorem*} 

We can summarise our method schematically as follows:  

{\small \begin{equation} \label{eq:method}
\begin{aligned}
\text{Killing} \\
\text{vector field}
\end{aligned}
\quad \rightsquigarrow \quad
\begin{aligned}
\text{Warped} \\
\text{product}
\end{aligned}
\quad \rightsquigarrow \quad
\text{PDE}
\quad \rightsquigarrow \quad
\begin{aligned}
\text{Riemannian} \\
\text{submersion}
\end{aligned}
\quad \rightsquigarrow \quad
\text{ODE}
\quad \rightsquigarrow \quad
\text{Soliton. }
\end{equation} }

This procedure is very general and can be used to produce many solitons, including those in~\cite{BL25}. To illustrate our technique, we study solitons on the hyperbolic space $\mathbb{H}^n$, where we find a new example of complete rotator -- see Section~\ref{sec:applications}.

Finally, we outline the structure of the paper:
\begin{itemize}
    \item In Section~\ref{sec:PDE}, we obtain the PDE that a function has to satisfy for its graph to be a soliton -- see Theorem~\ref{thm:pde}. 
    \item In Section~\ref{sec:ODE}, we obtain the ODE that results from imposing symmetry conditions on the soliton -- see Theorem~\ref{thm:ode}. 
    \item In Section~\ref{sec:applications}, we use our new method to obtain a new example of complete rotator in hyperbolic space. 
\end{itemize}

\section{The PDE}\label{sec:PDE}

Let $\Omega \subseteq M$ be open, let $u \colon \Omega \to I$ be a smooth function, and let $F \colon \Omega \to M \times_{\varphi} I$ be its graph $x \mapsto (x,u(x))$. In this section, we shall prove that $F$ is a soliton, i.e., it satisfies equation~\eqref{eq:mcf3},  if and only if $u$ satisfies a certain PDE~\eqref{eq:pde}. More specifically, we have the following result:

\begin{theorem}\label{thm:pde}
Let $(M,g_M)$ be an oriented Riemannian manifold, $\varphi \colon M \to (0,+\infty)$ a smooth function, $I \subseteq \R$ an interval and $(N,g) = (M \times I , g_M + \varphi dr^2)$, where $r$ is the usual coordinate in $\R$. Let $\mathcal{G}$ be the one-parameter family of isometries of $N$ given by $\mathcal{G}_t (x,r) = (x , r+t)$. Then, if $\Omega \subseteq M$ is open, the graph of a smooth function $u \colon \Omega \to I$ defines a $\mathcal{G}$-soliton on $N$ if and only if $u$ satisfies
\begin{equation}\label{eq:pde}
    \mathrm{div}\left(\frac{\nabla u}{W}\right)=\frac{1}{W}-\frac{1}{2W\varphi}g(\nabla \varphi,\nabla u) \, ,
\end{equation}
where the function $W \colon M \to (0,+\infty)$ is defined by
\begin{equation}\label{eq:defW}
    W = \sqrt{g_M(\nabla u , \nabla u) + \frac{1}{\varphi}} \, .
\end{equation}
\end{theorem}

\begin{proof}
The Killing vector field associated to the one-parameter family of isometries $\mathcal{G}$ is just the coordinate vector field $\partial_r$ corresponding to the interval $I$. It is easily verified that a unit normal vector field to $F$ is given by
\begin{equation*}
\nu = \frac{1}{W} \left( - \nabla u + \frac{1}{\varphi} \partial_r \right) \, .
\end{equation*}
Hence, the right-hand side of equation~\eqref{eq:mcf3} is
\begin{equation}\label{eq:Knu}
g \left( \partial_r , \nu \right) = \frac{1}{W} \, .
\end{equation}
Next, we need to compute the mean curvature $H$ of $F$. To this end, let $(e_1 , \dots , e_n)$ be a local $g_M$-orthonormal frame of $M$. The mean curvature $H$ is given by
\begin{equation*}
H = g \left( \vec{H} , \nu \right) = \sum_{i,j=1}^{n} \gamma^{ij} g \left( \nabla_{F_* e_i} F_* e_j , \nu \right) \, ,
\end{equation*}
where $\gamma^{ij}$ are the coefficients of the inverse of the matrix $\gamma_{ij} = g \left( F_* e_i , F_* e_j \right)$. By definition of $F$ as the graph of $u$, $F_* e_i = e_i + u_i \partial_r$, where $u_i = e_i(u)$. Note that $\nabla u = \sum_{i=1}^{n} u_i e_i$. Hence,
\begin{equation*}
    \gamma_{ij} = \delta_{ij} + \varphi u_i u_j \quad \text{and thus} \quad \gamma^{ij} = \delta_{ij} - \frac{1}{W^2} u_i u_j \, .
\end{equation*}
Moreover,
\begin{align*}
    \nabla_{F_* e_i} F_* e_j = \nabla_{e_i + u_i \partial_r} \left( e_j + u_j \partial_r \right) = \nabla_{e_i} e_j + u_{ij} \partial_r + u_j \nabla_{e_i} \partial_r + u_i \nabla_{\partial_r} e_j + u_i u_j \nabla_{\partial_r} \partial_r \, ,
\end{align*}
where $u_{ij} = e_i(u_j)$. Furthermore, using Koszul's formula and the obvious facts that $[\partial_r,e_i] = 0$ and $g(\partial_r,[e_i,e_j]) = 0$, we see that
\[
2g\left(\nabla_{\partial_r} e_i,e_j\right)
=\partial_r\left(g\left(e_i,e_j\right)\right)
+e_i(g\left(\partial_r,e_j\right))-e_j(g\left(\partial_r,e_i\right))
 =\partial_r\left( g\left(e_i,e_j\right)\right)=0 \, .
\]
and therefore $g(\nabla_{\partial_r} X,Y)=0$ for any vector fields $X,Y$ tangent to $M$ and independent of $r$. Hence, using also that $0 = \partial_r(\varphi) = \partial_r(g(\partial_r,\partial_r))=2g(\nabla_{\partial_r} \partial_r,\partial_r)$ and that the leaves of the foliation are totally geodesic, we obtain
\begin{align*}
   g& \left( \nabla_{F_* e_i} F_* e_j , \nu \right)  \\ =& g\left(\nabla_{e_i}e_j,\nu\right)
    +u_{ij}g\left(\partial_r,\nu\right)
    +u_j g\left(\nabla_{e_i} \partial_r,\nu\right)
    +u_i g\left(\nabla_{\partial_r} e_j,\nu\right)
    +u_i u_j g\left(\nabla_{\partial_r} \partial_r,\nu\right) \\
  =  & \frac{1}{\varphi W} g\left(\nabla_{e_i}e_j,\partial_r\right)
    -\frac{1}{W}g\left(\nabla_{e_i}e_j,\nabla u\right)  +\frac{u_{ij}}{W} \\
     & +\frac{u_j}{\varphi W}g\left(\nabla_{e_i}\partial_r,\partial_r\right)
    -\frac{u_j}{W}g\left(\nabla_{e_i}\partial_r,\nabla u\right)
     +\frac{u_i}{\varphi W}g\left(\nabla_{\partial_r}e_j,\partial_r\right) \\
    &-\frac{u_i}{W}g\left(\nabla_{\partial_r}e_j,\nabla u\right)
    + \frac{u_iu_j}{\varphi W}g(\nabla_{\partial_r}\partial_r,\partial_r) - \frac{u_iu_j}{W} g(\nabla_{\partial_r}\partial_r,\nabla u)\\
    = &\frac{-1}{W}g\left(\nabla_{e_i}e_j,\nabla u\right)+\frac{u_{ij}}{W}
    +\frac{u_j}{\varphi W}g\left(\nabla_{e_i}\partial_r,\partial_r\right)
    +\frac{u_j}{W}g\left(\partial_r,\nabla_{e_i}\nabla u\right)\\
    &+\frac{u_i}{\varphi W} g\left(\nabla_{\partial_r}e_j,\partial_r\right)-\frac{u_i}{W}g\left(\nabla_{\partial_r}e_j,\nabla u\right)-\frac{u_iu_j}{W}g\left(\nabla_{\partial_r} \partial_r,\nabla u\right)\\
    =& \frac{-1}{W}g\left(\nabla_{e_i}e_j,\nabla u\right)
    +\frac{u_{ij}}{W}+\frac{u_j}{\varphi W}g\left(\nabla_{e_i}\partial_r,\partial_r\right)
    +\frac{u_i}{\varphi W}g\left(\nabla_{\partial_r}e_j,\partial_r\right)
    -\frac{u_iu_j}{W}g\left(\nabla_{\partial_r} \partial_r,\nabla u\right).
    \end{align*}
As $\partial_r$ is a Killing vector field, we have that
\begin{equation}\label{eq:ci}
\varphi_i:=e_i(g(\partial_r,\partial_r))=2g\left(\nabla_{e_i}\partial_r,\partial_r\right)=-2g\left(\nabla_{\partial_r} \partial_r,e_i\right)=2g\left(\nabla_{\partial_r}e_i,\partial_r\right) \, .
\end{equation}
Moreover,
\[g\left(\nabla_{e_i}e_j,\nabla u\right)
=e_i(g(e_j,\nabla u))-g\left(\nabla_{e_i}\nabla u,e_j\right)=u_{ij}-g(\nabla_{e_i}\nabla u,e_j) \, , \]
obtaining
\begin{equation}
    g \left( \nabla_{F_* e_i} F_* e_j , \nu \right)=\frac{1}{W}g\left(\nabla_{e_i}\nabla u,e_j\right)+\frac{u_j\varphi_i}{2\varphi W}+\frac{u_i\varphi_j}{2\varphi W}-\frac{u_iu_j}{W}g\left(\nabla_{\partial_r} \partial_r,\nabla u\right).
\end{equation}
Putting everything together, we obtain
\begin{eqnarray*}
H&=& \sum_{i,j=1}^{n} \gamma^{ij} g \left( \nabla_{F_* e_i} F_* e_j , \nu \right)\\
&=&\frac{\mathrm{div}(\nabla u)}{W} -\sum^{n}_{i,j=1}\frac{1}{W^3}u_iu_jg\left(\nabla_{e_i}\nabla u,e_j\right)\\
&&+\frac{1}{W}\sum^{n}_{i,j=1}\delta_{ij}\left[-u_iu_jg\left(\nabla_{\partial_r} \partial_r,\nabla u\right)+\frac{u_j\varphi_i}{2\varphi}+\frac{u_i\varphi_j}{2\varphi}\right]\\
&&-\frac{1}{W^3}\sum^{n}_{i,j=1}\left[-u_i^2u_j^2g\left(\nabla_{\partial_r} \partial_r,\nabla u\right)+\frac{u_iu_j^2\varphi_i}{2\varphi}+\frac{u_i^2u_j\varphi_j}{2\varphi}\right]\\
&=&\frac{\mathrm{div}(\nabla u)}{W}
-\sum^{n}_{i,j=1}\frac{1}{W^3}u_iu_jg\left(\nabla_{e_i}\nabla u,e_j\right) +\frac{1}{W}\left[-\sum^{n}_{i=1}u_i^2g\left(\nabla_{\partial_r} \partial_r,\nabla u\right)+\sum^{n}_{i=1}\frac{u_i\varphi_i}{\varphi}\right]\\
&&-\frac{1}{W^3}\left[-|\nabla u|^4
g\left(\nabla_{\partial_r} \partial_r,\nabla u\right)+\frac{1}{\varphi}\sum^{n}_{i,j=1}u_i\varphi_iu_j^2\right]\\
&=&\frac{\mathrm{div}(\nabla u)}{W} -\sum^{n}_{i,j=1}\frac{1}{W^3}u_iu_jg\left(\nabla_{e_i}\nabla u,e_j\right)\\
&&+\frac{1}{W^3}\left[-W^2|\nabla u|^2g\left(\nabla_{\partial_r} \partial_r,\nabla u\right)+W^2\frac{g(\nabla \varphi,\nabla u)}{\varphi}\right]\\
&&-\frac{1}{W^3}\left[-|\nabla u|^4g\left(\nabla_{\partial_r} \partial_r,\nabla u\right)+\frac{g(\nabla \varphi,\nabla u)|\nabla u|^2}{\varphi}\right].
\end{eqnarray*}
Bearing in mind the definition of $W$~\eqref{eq:defW}, we get
\[ H = \frac{\mathrm{div}(\nabla u)}{W}-\sum^{n}_{i,j=1}\frac{1}{W^3}u_iu_jg(\nabla_{e_i}\nabla u,e_j)+\frac{1}{\varphi W^3}\left[-|\nabla u|^2g\left(\nabla_{\partial_r} \partial_r,\nabla u\right)+\frac{1}{\varphi}g(\nabla \varphi,\nabla u)\right].
\]
By~\eqref{eq:ci}, we see that $\nabla_{\partial_r} \partial_r=-\frac{1}{2}\nabla \varphi$. Hence,
\begin{equation} \label{eq:nH}
H=\frac{\mathrm{div}(\nabla u)}{W} -\sum^{n}_{i,j=1}\frac{1}{W^3}u_iu_jg(\nabla_{e_i}\nabla u,e_j)+\frac{1}{\varphi W^3}g(\nabla \varphi,\nabla u)\left(\frac{1}{2}|\nabla u|^2+\frac{1}{\varphi}\right).
\end{equation}
Notice now that, if $W_i = e_i (W)$,
\begin{eqnarray*}
\frac{\mathrm{div}(\nabla u)}{W}&=&
\mathrm{div}\left(\frac{\nabla u}{W}\right) -g\left(\nabla\left(\frac{1}{W}\right),\nabla u\right) =\mathrm{div}\left(\frac{\nabla u}{W}\right)+\sum_{i=1}^n\frac{W_i}{W^2}u_i\\
&=&\mathrm{div}\left(\frac{\nabla u}{W}\right)
+\sum_{i=1}^n  \frac{1}{2W^3}\left( -\frac{\varphi_i}{\varphi^2}+2g(\nabla_{e_i}\nabla u,\nabla u)\right)u_i\\
&=&\mathrm{div}\left(\frac{\nabla u}{W}\right) +\sum_{i,j=1}^n\frac{1}{W^3}u_iu_jg(\nabla_{e_i}\nabla u,e_j)-\frac{1}{2\varphi^2W^3}g(\nabla \varphi,\nabla u).
\end{eqnarray*}
By using equation~\eqref{eq:nH}, we obtain
\begin{eqnarray*}
H=\mathrm{div}\left(\frac{\nabla u}{W}\right)+\frac{1}{W^3}\left(\frac{1}{2\varphi}|\nabla u|^2+\frac{1}{2\varphi^2}\right)g(\nabla \varphi,\nabla u) \, ,
\end{eqnarray*}
and therefore
\begin{eqnarray}\label{H_Killinggraph}
H=\mathrm{div}\left(\frac{\nabla u}{W}\right)+\frac{1}{2W\varphi}g(\nabla \varphi,\nabla u) \, .
\end{eqnarray}
This, together with~\eqref{eq:Knu}, yields the result.
\end{proof}
\begin{remark} Some comments about other works are in order. First, our study is a generalisation of Lira and Mart{\'i}n's work,~\cite{LM}, where they studied translators in Riemannian products $(M\times\mathbb{R},g_M+ dt^2)$ which move in the direction of $\mathbb{R}$, because our setting reduces to theirs when $\varphi$ is a positive constant function. Second, a formally equal formula to~\eqref{H_Killinggraph} can be obtained in a semi-Riemannian environment by adjusting the corresponding definition of $W$, where some constant functions $\varepsilon,\widetilde{\varepsilon}=\pm 1$ must be considered, in a similar way as in~\cite{LO}. 
\end{remark}

\section{The ODE} \label{sec:ODE}

A geometric way to find solutions to the PDE~\eqref{eq:pde} is to look for functions $u \colon M \supseteq \Omega \to I$ that have some sort of symmetry. For example, when studying translating solitons in Euclidean space $\R^3 = \R^2 \times \R$ that are graphs of functions $u \colon \R^2 \to \R$, one tries to find rotationally invariant solutions to~\eqref{eq:pde}. This is an instance of a more general procedure by which one looks for solutions that are invariant under a cohomogeneity-one action of a Lie group on $M$ by isometries. Generalising this approach, it is natural to look for solutions that factor through a certain \emph{Riemannian submersion}. 

\begin{definition}A manifold $M\times_{\varphi} I$ will be called \emph{submersive} if there exist a Riemannian submersion $\pi:(M,g)\to (J,\alpha(s)ds^2)$, $\alpha:J\to\R^+$, and smooth functions $\hat{\varphi}:J\to I$ and $h:J\to\R$ such that $\varphi=\hat{\varphi}\circ\pi$ and for each $s\in J$, $\pi^{-1}\{s\}$ is a hypersurface in $M$ with constant mean curvature $h(s)$ with respect to $-\nabla\pi/\|\nabla\pi\|$.   
\end{definition} 
Next, given an open subinterval $J'\subset J$, consider a smooth function $f:J'\to I$,  so that we obtain the following commutative diagrams:

\begin{equation*}\label{diagrams}
    \xymatrix{ M \supset \pi^{-1}(J')  \ar[d]^{\pi} \ar[dr]^{u} \\ (J \supseteq J' , \alpha(s) ds^2) \ar[r]_-{f} & I } \qquad \qquad
    \xymatrix{M \supset \pi^{-1}(J')  \ar[r]^-{\mathrm{gr}(u)}  \ar[d]^{\pi}   &
    M\times_{\varphi} I = N \ar[d]^{\pi\times \mathrm{Id}} \\
    J \supseteq J' \ar[r]^-{\mathrm{gr}(f)}  & J \times_{\widehat{\varphi}} I }
    \end{equation*}
where $\mathrm{gr}(u)$ and $\mathrm{gr}(f)$ are the graphs of $u$ and $f$ respectively, we obtain the following general result:

\begin{theorem} \label{thm:ode}
%
Assume that the manifold $M\times_{\varphi}I$ is submersive, and such that the previous commutative diagrams concerning $f$ and $u=f\circ\pi$ are satisfied. Then, the graph $\mathrm{gr}(u)$  is a $\mathcal{G}$-soliton on $N$ if and only if $f$ satisfies the following ODE:
\begin{equation}\label{eq:ode}
        f''=\left(\alpha+\widehat{\varphi}(f')^2\right)\left( 1-\frac{\widehat{\varphi}'f'}{ 2\widehat{\varphi}\alpha} - \frac{f'h}{\sqrt{\alpha}}\right)
+\frac{f'}{2}\left(\log\left(\frac{\alpha}{\widehat{\varphi}}\right)\right)' \, .
\end{equation}
\end{theorem}

\begin{remark} Note that the warping function $\varphi$ essentially determines the projection $\pi$, since $\varphi$ has to factor through it to allow our technique to work. 
\end{remark}
\begin{remark} This result applies to general warped products $N = M \times_{\varphi} I$. In particular, it applies whenever $N$ has a nowhere-vanishing Killing vector field with integrable orthogonal distribution, as discussed in Section~\ref{sec:intro}.
\end{remark}

\begin{proof}
This result is a consequence of Theorem~\ref{thm:pde}, and equation~\eqref{eq:ode} is obtained by substituting $u = f \circ \pi$ in~\eqref{eq:pde}. We now write the details of this computation. First, note that
\[ \nabla u = \nabla (f \circ \pi) = (f' \circ \pi)  \nabla \pi \, , \]
and thus, using that $\pi \colon (M,g_M) \to (J , \alpha(s) ds^2)$ is a Riemannian submersion, we get
\[ g_M \left( \nabla u , \nabla u \right) = (f' \circ \pi)^2 g_M \left( \nabla \pi , \nabla \pi \right) = \frac{(f' \circ \pi)^2}{(\alpha \circ \pi)} \, . \]
Hence, substituting this into equation~\eqref{eq:defW}, we get
\[ W = \sqrt{\frac{(f' \circ \pi)^2}{(\alpha \circ \pi)} + \frac{1}{(\widehat{\varphi} \circ \pi)}} = \widehat{W} \circ \pi \, , \quad \text{where} \quad \widehat{W} = \sqrt{\frac{(f')^2}{\alpha} + \frac{1}{\widehat{\varphi}}} \, . \]
A straightforward computation using the definition of $\widehat{W}$ yields
\[ \widehat{W}' = \frac{f' f''}{\alpha \widehat{W}} - \frac{(f')^2 \alpha'}{2 \alpha^2 \widehat{W}} - \frac{\widehat{\varphi}'}{2 \widehat{W} \widehat{\varphi}^2} \, . \]
Moreover, if $(e_1 , \dots e_n)$ is a local $g_M$-orthonormal frame of $M$ with $e_n = - \sqrt{\alpha \circ \pi} \, \nabla \pi$,
\begin{align*}
\div{\nabla \pi} &= \sum_{i=1}^{n} g_M \left( \nabla_{e_i} \nabla \pi , e_i \right) = - \sum_{i=1}^{n-1} g_M \left( \nabla \pi , \nabla_{e_i} e_i \right) + g_M \left( \nabla_{-\sqrt{\alpha \circ \pi} \nabla \pi} \nabla \pi , -\sqrt{\alpha \circ \pi} \nabla \pi \right) \\
&= \frac{1}{\sqrt{\alpha}} \sum_{i=1}^{n-1} g_M \left( - \sqrt{\alpha \circ \pi} \nabla \pi , \nabla_{e_i} e_i \right)  + (\alpha \circ \pi) g_M \left( \nabla_{\nabla \pi} \nabla \pi , \nabla \pi  \right) \\
&= \frac{(h \circ \pi)}{\sqrt{(\alpha \circ \pi)}} + \frac{1}{2}(\alpha \circ \pi) \nabla \pi \left( g_M \left( \nabla \pi , \nabla \pi \right) \right) = \frac{(h \circ \pi)}{\sqrt{(\alpha \circ \pi)}} + \frac{1}{2}(\alpha \circ \pi) \nabla \pi \left( \frac{1}{(\alpha \circ \pi)} \right) \\
&= \frac{(h \circ \pi)}{\sqrt{(\alpha \circ \pi)}} + \frac{1}{2}(\alpha \circ \pi) g_M \left( \nabla \pi , \nabla \left(\frac{1}{(\alpha \circ \pi)}\right) \right) = \frac{(h \circ \pi)}{\sqrt{(\alpha \circ \pi)}} - \frac{(\alpha' \circ \pi)}{2 (\alpha \circ \pi)^2} \, .
\end{align*}
Finally,  
\begin{align*}
\div \left( \frac{\nabla u}{W} \right) &= \div \left( \frac{(f' \circ \pi)}{W} \nabla \pi \right) = \frac{(f' \circ \pi)}{W} \div \left( \nabla \pi \right) + g_M \left( \nabla \pi , \nabla \left(\frac{(f' \circ \pi)}{W} \right) \right) \\
&= \frac{(f' \circ \pi)}{W} \div \left( \nabla \pi \right) + \frac{(f'' \circ \pi)W - (f' \circ \pi) (\widehat{W}' \circ \pi) }{W^2} g_M \left( \nabla \pi , \nabla \pi \right) \\
&= \frac{(f' \circ \pi)}{W} \div \left( \nabla \pi \right) + \frac{(f'' \circ \pi)}{W (\alpha \circ \pi)} - \frac{(f' \circ \pi)(\widehat{W}' \circ \pi)}{W^2 (\alpha \circ \pi)} \, ,
\end{align*}

and putting everything together we get that the left-hand side of the PDE~\eqref{eq:pde} for $u = f \circ \pi$ now reads
\begin{equation} \label{eq:proof_lhs}
\div \left( \frac{\nabla u}{W} \right) = \left( \frac{f'h}{\widehat{W} \sqrt{\alpha}} - \frac{f' \alpha'}{2 \widehat{W} \alpha^2} + \frac{f''}{\widehat{W} \alpha} \left( 1 - \frac{(f')^2}{\widehat{W}^2 \alpha} \right) + \frac{(f')^3 \alpha'}{2 \widehat{W}^3 \alpha^3} + \frac{f' \widehat{\varphi}'}{2 \widehat{W}^3 \alpha \widehat{\varphi}^2}  \right) \circ \pi \, .
\end{equation}
Similarly, the right-hand side of~\eqref{eq:pde} is now
\begin{equation} \label{eq:proof_rhs}
\frac{1}{W} - \frac{g_M \left( \nabla \varphi , \nabla u \right)}{2 W \varphi} = \frac{1}{W} - \frac{f' \widehat{\varphi}' g_M \left( \nabla \pi , \nabla \pi\right)}{2 (\widehat{W} \circ \pi)(\widehat{\varphi} \circ \pi)} = \left(\frac{1}{\widehat{W}} - \frac{f' \widehat{\varphi}'}{2 \widehat{\varphi} \widehat{W} \alpha} \right) \circ \pi \, .
\end{equation}
Finally, equating~\eqref{eq:proof_lhs} and~\eqref{eq:proof_rhs}, we obtain the ODE~\eqref{eq:ode}.
\end{proof}

\begin{remark}
Note that, if $\alpha$ and $\widehat{\varphi}$ are constant equal to $1$, equation~\eqref{eq:pde} reduces to the usual mean curvature flow equation in the Riemannian product $(M \times \mathbb{R},g + dr^2)$, and equation~\eqref{eq:ode} to the ODE found in~\cite{LO}.
\end{remark}

\section{An Application to the Hyperbolic Space} \label{sec:applications}

We now use the technique~\eqref{eq:method} to obtain new examples of solitons of the mean curvature flow. For $n \geq 2$, consider the half-space model of hyperbolic space
    \[\H{n+1}=\left\{ x=(x_1,\ldots,x_{n} , x_{n+1})\in\R^{n+1} : x_1>0\right\}, \qquad g=\frac{1}{x_1^2}\langle \cdot , \cdot \rangle \, ,
    \]
    where $\langle \cdot , \cdot \rangle$ is the standard flat metric of Euclidean space. Consider the one-parameter family of isometries $\mathcal{G}_t:\mathbb{H}^{n+1}\to\mathbb{H}^{n+1}$, $t\in\mathbb{R}$, given by
    \begin{equation*}
        \mathcal{G}_{t} (x_1 , \dots , x_{n} , x_{n+1}) = (x_1 , \dots , x_{n-1} , \mathrm{Rot}_{t} \left(x_{n} , x_{n+1} \right)) \, , \\
    \end{equation*}
    where $\mathrm{Rot}_t$ denotes the rotation in $\R^2$ by an oriented angle $t$. The Killing vector field corresponding to $ \mathcal{G}$ via~\eqref{eq:def_K} is 
    $$K=-x_{n+1}\partial_n+x_{n}\partial_{n+1} \, . $$
    There is a warped product structure on $\mathbb{H}^{n+1} \setminus \{x_n = x_{n+1} = 0 \}$, namely 
    \begin{equation} \label{eq:iso}
       \mathbb{H}^{n+1} \setminus \{x_n = x_{n+1} = 0 \} \cong \check{\mathbb{H}}^n \times_{\varphi} S^1 \, ,
    \end{equation}
    where $\check{\mathbb{H}}^n = \left\{ x \in \mathbb{H}^n \colon x_{n} > 0 \right\}$ and $\varphi \colon \check{\mathbb{H}}^n \to (0,+\infty)$ is given by $x \mapsto \left( x_n / x_1 \right)^2$. The isometry~\eqref{eq:iso} is given by
    \begin{equation*}
    \left( (x_1, \dots, x_{n-1}, x_n), e^{i \theta} \right) \mapsto (x_1, \dots, x_{n-1}, x_n \cos(\theta), x_n \sin(\theta)) \, .
    \end{equation*}

    In the general theory, the second term in the warped product was a real interval, and now we have $S^1$. But smooth maps from an open interval to $S^1$ correspond to maps from that interval to $\R$, by standard covering theory.

    Now, take the submersion $\pi \colon \check{\mathbb{H}}^n \to (0, +\infty)$ given by $\pi(x) = x_n / x_1$. Note that $\varphi$ factors through $\pi$ as $\varphi = \widehat{\varphi} \circ \pi$, where $\widehat{\varphi}(s) = s^2$.  This submersion satisfies $g \left( \nabla \pi , \nabla \pi \right) = 1/\left(\alpha \circ \pi\right)$, where $\alpha \colon (0,+\infty) \to (0 , +\infty)$ is given by
    \begin{equation}\label{eq:alpha} 
    \alpha(s) = \frac{1}{1+s^2} \, ,   
    \end{equation}
    as shown in Appendix~\ref{appendix:mean_curvature}. Observe that $\pi$ has constant mean curvature fibres. Indeed, the mean curvature of the fibre $\pi^{-1 }(\{s\}) \subseteq \check{\mathbb{H}}^n$, $s>0$, with respect to the unit normal vector field $- \nabla \pi / \lVert \nabla \pi \rVert$ is given by the function $h \colon (0,+\infty) \to \R$,
     \begin{equation} \label{eq:mean_curvature}
         h(s) = (n-1) \frac{s}{\sqrt{1+s^2}} \, , 
     \end{equation} 
    which is derived in Appendix~\ref{appendix:mean_curvature}. Finally, putting everything together, ODE~\eqref{eq:ode} in this case reads 
    \begin{equation} \label{ode2}
        f''(s) = \frac{\big(\left(s^4+s^2\right) f'(s)^2+1\big) \big(s-\left(n s^2+1\right) f'(s)\big)-\left(2 s^2+1\right) f'(s)}{s(1+s^2)} \, . \\ 
    \end{equation}

We will show that there exists a one-parameter family of smooth complete $\mathcal{G}$-solitons on $\mathbb{H}^{n+1}$ that can be constructed from solutions to~\eqref{ode2}. These examples are new. In particular, they are different from the ones exhibited in~\cite{LRS}. In fact, our  $\mathcal{G}$-solitons on $\mathbb{H}^{n+1}$ can be constructed from solutions $f \colon (s_1 , s_2) \to \mathbb{R}$ to~\eqref{ode2}, which can be parametrised by
\[
\check{\mathbb{H}}^{n} \ni \left(y_1 , \dots , y_{n-1}, y_n \right) \mapsto \left( y_1 , \dots , y_{n-1}, y_n \cos\left( f(y_n/y_1) \right), y_n \sin\left( f(y_n/y_1) \right) \right) \in \mathbb{H}^{n+1} \, .
\]
Note that each subset $\mathcal{H}^{n-2}(c_1,c_2)=\{(y_1,\ldots,y_n)\in \mathbb{H}^{n} \mid y_1=c_1, \, y_n=c_2\}$, $c_1>0$ and $c_2\in\mathbb{R}$, is a $(n-2)$-dimensional horosphere. Indeed, $y_1=c_1$ defines a $n$-dimensional horosphere, and $y_n=c_2$ is a totally geodesic hyperplane. Thus, the intersection is a $(n-2)$-horosphere. And they are moved by this map to a $(n-2)$-horosphere in $\mathbb{H}^{n+1}$. Thus, all these $n$-dimensional examples are foliated by $(n-2)$-horospheres.
Defining $u = y_1$, $t_i = y_i / y_1$ for $i = 2 , \dots , n-1$ and $v = y_n / y_1$, we get a reparametrisation
\[
\check{\mathbb{H}}^{n} \ni \left(u , t_2 , \dots , t_{n-1}, v \right) \mapsto u \left( 1 , t_2 , \dots , t_{n-1}, v \cos\left( f(v) \right), v \sin\left( f(v) \right) \right) \in \mathbb{H}^{n+1} \, .
\]

Hence, the soliton is determined by the curve $v \mapsto \left( v \cos\left( f(v) \right), v \sin\left( f(v) \right) \right)$. 

To study~\eqref{ode2}, we define $w=f'$, $p(s)=s^3+s$, and $q(s,x)=\big((s^4+s^2)x^2+1\big)\big(s-(ns^2+1)x\big)-(2s^2+1)x$. In this way,~\eqref{ode2} is equivalent to
\begin{equation}\label{18bis} f'=w, \quad w'(s)=\frac{q(s,w(s))}{p(s)}. 
\end{equation}
Note that $p(0)=0$ and $q(0,0)=0$. 

Firstly, we show that solutions to~\eqref{18bis} can always be extended to infinity. 

\begin{lemma} \label{extendtoinfinity} Let $f\colon(s_0,s_1)\to \R$, $0<s_0<s_1$, be a solution to~\eqref{ode2}. Then, it can be extended (as a solution) to $f\colon (s_0,+\infty) \to \R$.
\end{lemma}
\begin{proof} If we recall~\eqref{18bis}, it is enough to show that a solution $w$ to the second equation can be extended to $(s_0,+\infty)$. A simple computation shows that
\[q(s,x)= -s^2(s^2+1)(ns^2+1)x^3 +s^3(s^2+1)x^2-((n+2)s^2+2)x+s.
\]
It is easy to see that for fixed $s>0$, $q(s,x)\to -\infty$ when $x\to +\infty$ and $q(s,x)\to +\infty$ when $x\to -\infty$. Then, there exist $x^+(s)>0$ and $x^-(s)<0$ such that $q(s,x)\le -1$ for any $x\ge x^+(s)$ and $q(s,x)\ge 1$ for any $x\le x^-(s)$.

Now, given a solution $w:(s_0,s_1)\to\R$, assume by contradiction that $s_1$ is a finite-time blow-up for $w$. Let
\begin{align*}
x^+ &:= \sup\{x^+(s) \mid s\in [(s_0+s_1)/2, s_1]\},\\
x^- &:= \inf\{x^-(s) \mid s\in [(s_0+s_1)/2, s_1]\},
\end{align*}
which are finite because $q$ is a polynomial and $[(s_0+s_1)/2, s_1]$ is compact. There are two cases.

First, $w(s)\to+\infty$ when $s\to s_1$, $s<s_1$. In particular, $w'(s)\to +\infty$ when $s\to s_1$. Then, there exists a sequence $\{r_n\}\subset [(s_0+s_1)/2,s_1)$ such that $\{r_n\}\nearrow s_1$ and $w(r_n)\to+\infty$. Therefore, there exists $m\in\mathbb{N}$ such that, for all $n\ge m$, $w(r_n)> x^+$, which implies $w'(r_n)=q(r_n,w(r_n))/p(r_n)<-1/p(r_n)<0$. This is a contradiction. 

Second, $w(s)\to-\infty$ when $s\to s_1$, $s<s_1$. In particular, $w'(s)\to -\infty$ when $s\to s_1$. 
Then, there exists a sequence $\{r_n\}\subset [(s_0+s_1)/2,s_1)$ such that $\{r_n\}\nearrow s_1$ and $w(r_n)\to-\infty$. Therefore, for some $m\in\mathbb{N}$, if $n\ge m$, then $w(r_n)< x^-$, which implies $w'(r_n)=q(r_n,w(r_n))/p(r_n)>1/p(r_n)>0$ for any $n\ge m$.  This is a contradiction. 
\end{proof}
\begin{lemma}\label{ff2} Given $f_0\in\R$, there exists a unique solution to ~\eqref{18bis}, $\ff^{(1)}\in \mathcal{C}^{\infty}[0,+\infty)$, such that $\ff^{(1)}(0)=f_0$ and $\big(\ff^{(1)}\big)'(0)=0$.
\end{lemma}
\begin{proof} Once we solve the problem $w'(s)=q(s,w(s))/p(s)$ with $w'(0)=0$, we just need to take $\ff^{(1)}(s)=\int_0^s w(\xi)d\xi+f_0$. We now consider the dynamical system $X:\R^2\to\R^2$, $X(s,x)=(p(s),q(s,x))$. Note that $X(0,0)=(0,0)$, $p'(0)=1$ and $\frac{\partial q}{\partial x}(0,0)=-2$, and apply Lemma~\ref{boundarysol}. Observe that the functions $p$ and $q$ are of class $\mathcal{C}^{\infty}$, and hence $\ff^{(1)}$ will also be of class $\mathcal{C}^{\infty}$ on a suitable interval. Finally, apply Lemma~\ref{extendtoinfinity}. 
\end{proof} 

The solution $\ff^{(1)}$ of Lemma~\ref{ff2}, when glued with its rotation by an angle $\pi$, provides a new complete smooth $\mathcal{G}$-soliton on $\mathbb{H}^{n+1}$, which can be seen for $n=2$ in Figures~\ref{figure:curve_bowl} and~\ref{figure:soliton_bowl}. If we denote by $F$ the embedding of this soliton into $\mathbb{H}^{n+1}$, then, by construction, its rotation $\{\mathcal{G}_{t} \circ F\}_{t \in \mathbb{R}}$ evolves by mean curvature flow. 

\begin{figure}[ht]
\centering\includegraphics[scale=0.6]{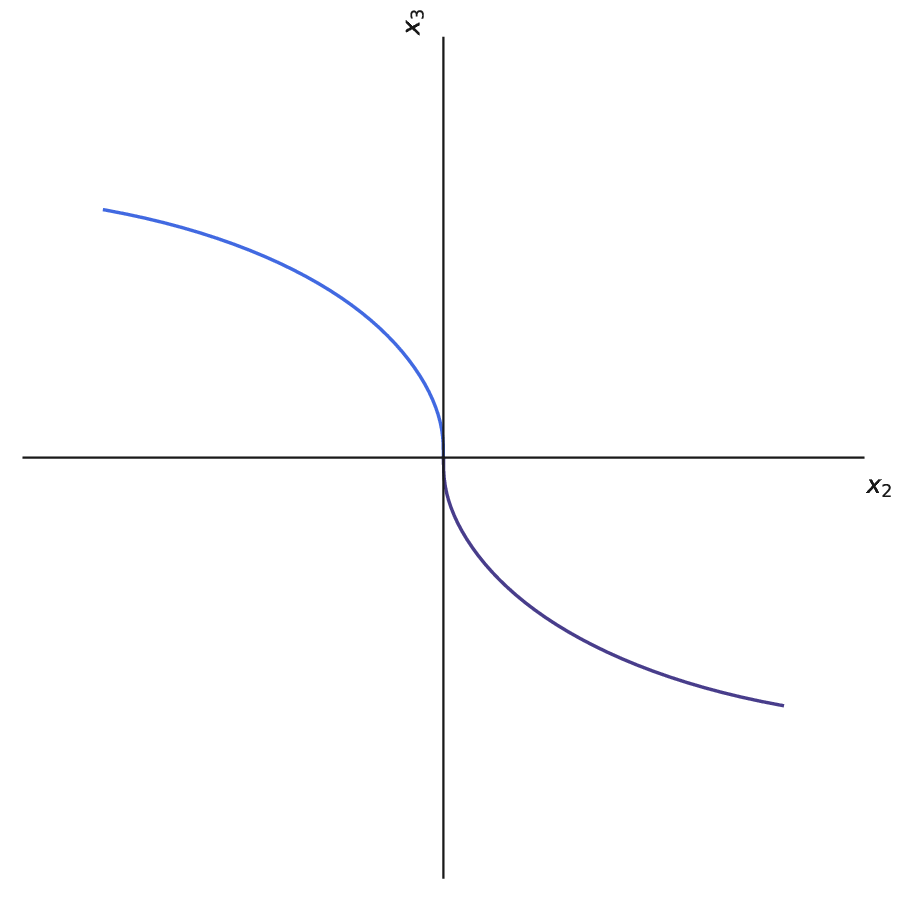}
\caption{Intersection of our new $\mathcal{G}$-soliton on $\mathbb{H}^3$ with the plane $x_1 = 1$. }\label{figure:curve_bowl}
\end{figure}

\begin{figure}[ht]
\centering\includegraphics[scale=0.8]{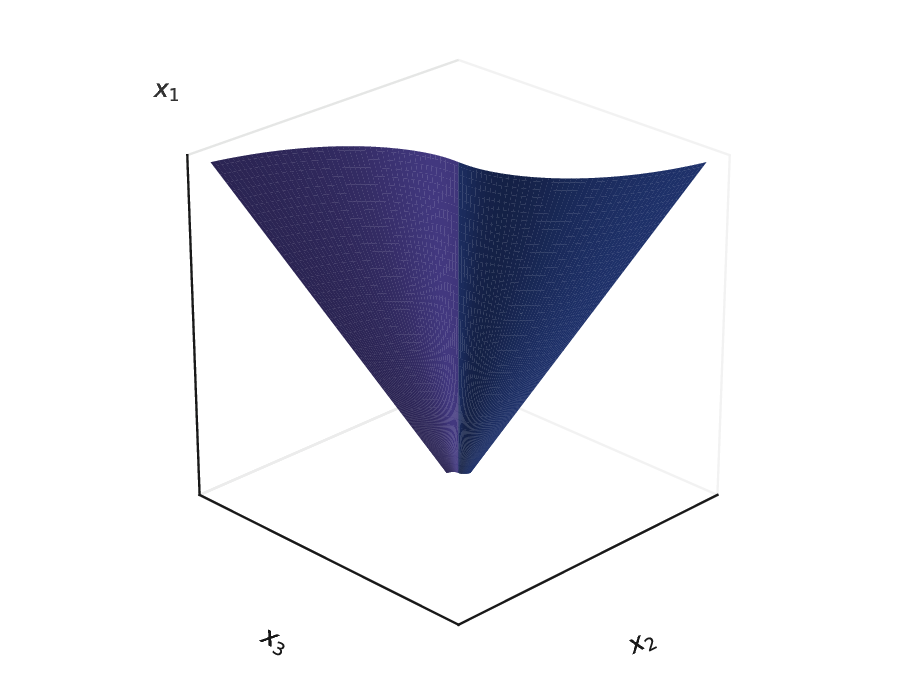}
\caption{A new complete $\mathcal{G}$-soliton on $\mathbb{H}^3$. }\label{figure:soliton_bowl}
\end{figure}

\begin{lemma}\label{wing-likesol} Given $s_0\in (0,+\infty)$, $r_0\in\R$, there exist $\ff_{\pm}^{(2)}\in \mathcal{C}^0[s_0,+\infty)\cap \mathcal{C}^{\infty}(s_0, +\infty)$ which are solutions to~\eqref{ode2} such that $\ff^{(2)}_{+}(s_0)=\ff^{(2)}_{-}(s_0)=r_0$, and $\lim_{s\to s_0}\big(\ff_{\pm}^{(2)}\big)'(s)=\pm\infty$.
\end{lemma}
\begin{proof} We use the standard technique of the inverse map. Indeed, assume that $f$ is a injective solution to~\eqref{ode2}, and let $\gamma$ be a local inverse. Then, given $r$ in some interval, it holds $r=f(\gamma(r))$, so that $1=f'(\gamma)\gamma'$, that is $0=f''(\gamma)(\gamma')^2+f'(\gamma)\gamma''$. From here, using $f'(\gamma)=1/\gamma'$, we obtain
\begin{equation}\label{inverse-f}
\gamma''(r)=\frac{(2\gamma(r)^2+1)\gamma'(r)^2-\big(\gamma(r)^4+\gamma(r)^2+\gamma'(r)^2\big)\big(\gamma(r)\gamma'(r)-n\gamma(r)^2-1\big)}{\gamma(r)^3+\gamma(r)}.
\end{equation}
Now, given $r_0=f(s_0)$, $\gamma(r_0)=s_0$, we impose $\gamma'(r_0)=0$. Inserting them into~\eqref{inverse-f}, we obtain
\[ \gamma''(r_0)= s_0(n s_0+1)>0.\] 
This means that $\gamma$ is convex in a small neighbourhood of $r_0$, say $(r_0-\rho_0,r_0+\rho_0)$ for some $\rho_0>0$. The restrictions of $\gamma$ to the intervals $(r_0-\rho_0,r_0]$ and $[r_0,r_0+\rho_0)$ are injective, and then they provide two inverse functions $\ff_{\pm}^{(2)}:[0,\delta)\to\R$ as desired, since $\lim_{s\to s_0}\big(\ff_{\pm}^{(2)}\big)'(s)=\lim_{r\to r_0}\frac{1}{\gamma'(r)}=\pm\infty$.
\end{proof}
Consider two solutions $\mathbf{f}^{(2)}_{\pm}$ as in Lemma~\ref{wing-likesol}. Similarly to the wing-like solitons obtained in~\cite{CSS}, the union of the two associated $\mathcal{G}$-solitons becomes a single smooth hypersurface in $\mathbb{H}^{n+1}$, which we will call a \emph{wing-like}-$\mathcal{G}$-soliton -- see Figures~\ref{figure:curve_winglike} and~\ref{figure:soliton_winglike}. 

\begin{figure}[ht]
\centering\includegraphics[scale=0.6]{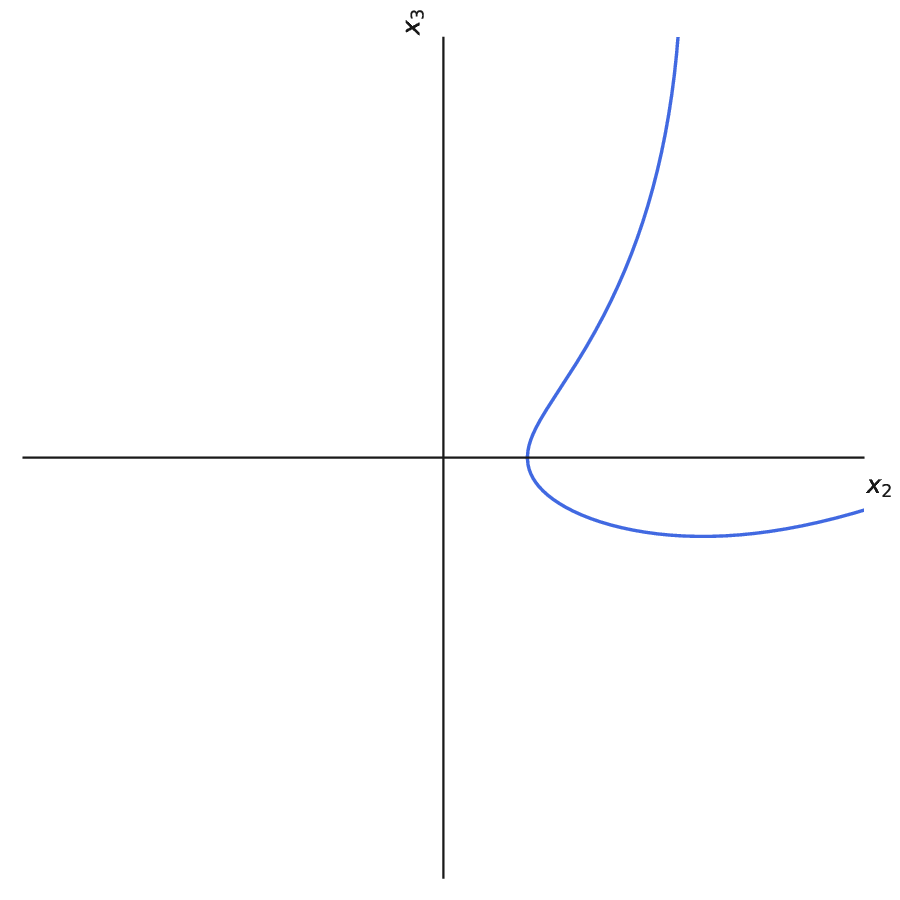}
\caption{Intersection of a wing-like $\mathcal{G}$-soliton with the plane $x_1 = 1$. }\label{figure:curve_winglike}
\end{figure}

\begin{figure}[ht]
\centering\includegraphics[scale=0.8]{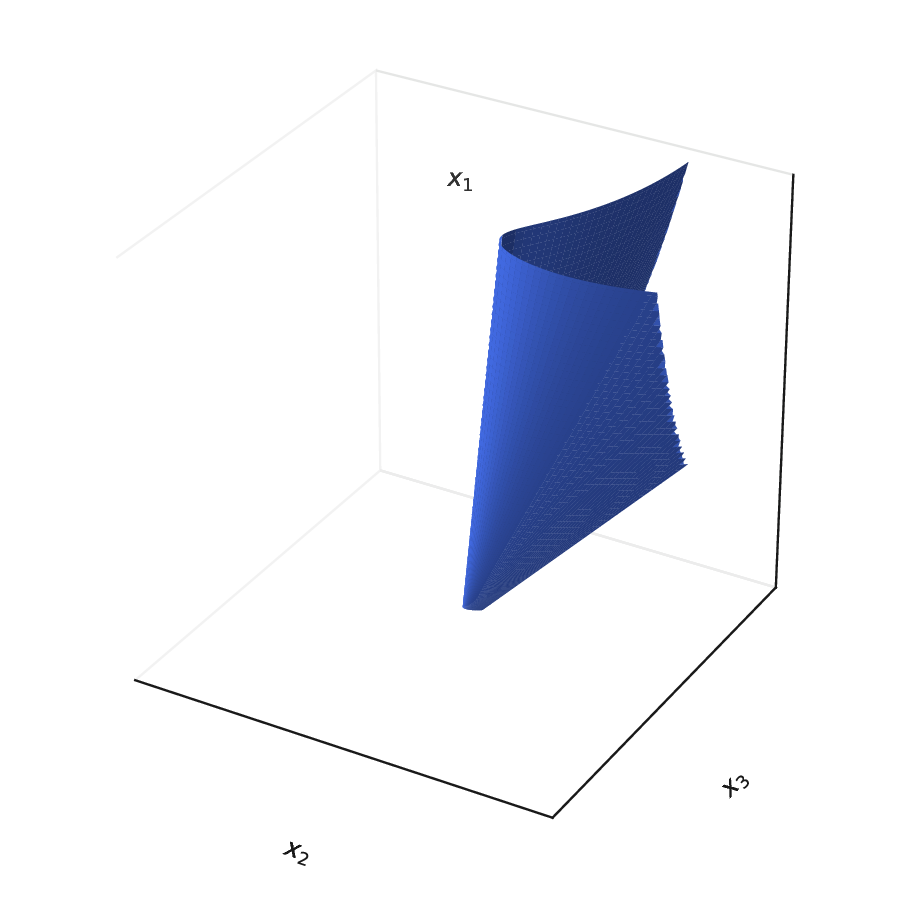}
\caption{A complete wing-like $\mathcal{G}$-soliton on $\mathbb{H}^3$. }\label{figure:soliton_winglike}
\end{figure}

\begin{remark}
It is relatively straightforward to show that, apart from the Killing vector field induced by rotations we used in this section, there are exactly two other Killing vector fields with integrable horizontal distribution in $\mathbb{H}^n$, namely the ones associated via~\eqref{eq:def_K} to the following one-parameter families of isometries:
    \begin{align*}
        &\mathcal{G}^{(1)}_{t} (x_1 , \dots , x_{n} , x_{n+1}) = (x_1 , \dots , x_{n} , x_{n+1} + t) \, , \\
        &\mathcal{G}^{(2)}_{t} (x_1 , \dots , x_{n} , x_{n+1}) = e^t \, (x_1 , \dots , x_{n} , x_{n+1}) \, ,
    \end{align*}
    The Killing vector fields corresponding to these isometries (respectively translations and dilations) are given by
    $$K^{(1)}=\partial_{n+1}  \, \quad K^{(2)} = \sum_{i=1}^{n}x_i \partial_i \, .  $$
    Each of these Killing vector fields induces a warped product structure on a certain open subset $N^{(i)} \subseteq \mathbb{H}^{n+1}$:
    \begin{equation} \label{eq:iso2}
        N^{(i)} \cong M^{(i)} \times_{c^{(i)}} I^{(i)} \, ,
    \end{equation}
    where $\left( M^{(i)} , g_{M^{(i)}} \right)$ is an oriented $n$-dimensional Riemannian manifold, $I^{(i)} \subseteq \R$ is an open interval and $\varphi^{(i)} \colon M^{(i)} \to (0,+\infty)$ is smooth. More specifically, we can take
    \begin{align*}
        &M^{(1)} = \mathbb{H}^{n} \, , \quad I^{(1)} = \R \, , \quad c^{(1)} (x) = \frac{1}{x_1^2} \, , \\
        &M^{(2)} = \mathbb{S}_+^n \setminus \{ (1,0, \dots ,0) \}  \, , \quad I^{(2)} = (0,+\infty)  \, , \quad c^{(2)} (x) = \frac{1}{x_1^2} \, ,
    \end{align*}
    where $\check{\mathbb{H}}^n = \left\{ x \in \mathbb{H}^n \colon x_{n} > 0 \right\}$ and $\mathbb{S}^n_+ = \left\{ x \in \mathbb{H}^{n+1} \colon \sum_{i=1}^{n+1} x_i^2 = 1 \right\}$, both equipped with the induced metrics from hyperbolic space. In fact, $\mathbb{S}^n_+$ is a totally geodesic hyperplane. The isometries~\eqref{eq:iso2} are given by
    \begin{align*}
    &M^{(1)} \times I^{(1)} \to  N^{(1)} \, , &\left( (x_1, \dots, x_n), x_{n+1} \right) &\mapsto (x_1, \dots, x_{n+1}) \, ; \\
    &M^{(2)} \times I^{(2)}  \to N^{(2)} \, , &\left( (x_1, \dots, x_{n+1}), r \right) &\mapsto r (x_1, \dots, x_{n+1}) \, .
\end{align*}
In each case, there exists a submersion $\pi^{(i)} \colon M^{(i)} \to J^{(i)}$, where $J^{(i)} \subseteq \R$ is an interval, through which $\varphi^{(i)}$ factors as $\varphi^{(i)} = \widehat{\varphi}^{(i)} \circ \pi^{(i)}$. This submersion is unique up to composition by diffeomorphisms $J^{(i)} \to \widetilde{J}^{(i)} \subseteq \R$. We can take
    \begin{align*}
        &J^{(1)} = (0,+\infty) \, , \quad \pi^{(1)}(x) = x_1 \, , \quad \widehat{\varphi}^{(1)} (s) = \frac{1}{s^2} \, , \\
        &J^{(2)} = (0,1) \, , \quad \pi^{(2)}(x) = x_1 \, , \quad \widehat{\varphi}^{(2)} (s) = \frac{1}{s^2} \, .\\
    \end{align*}
    These submersions satisfy $g_{M^{(i)}} \left( \nabla \pi^{(i)} , \nabla \pi^{(i)} \right) = 1/\left(\alpha^{(i)} \circ \pi^{(i)}\right)$, where the $\alpha^{(i)} \colon J^{(i)} \to (0 , +\infty)$ are easily verified to be given by
    \[ \alpha^{(1)}(s) = \frac{1}{s^2} \, , \quad \alpha^{(2)}(s) = \frac{1}{1-s^2} \, .
    \]
    Each $\pi^{(i)}$ has constant mean curvature fibres. Using similar calculations to the ones in Appendix~\ref{appendix:mean_curvature}, it is possible to compute the mean curvature of the fibre $\left(\pi^{(i)}\right)^{-1}(\{s\})\subseteq M^{(i)}$, $s\in J^{(i)}$, with respect to the unit normal vector field $- \nabla \pi^{(i)} / \lVert \nabla \pi^{(i)} \rVert$, which  is given by the function $h^{(i)} \colon J^{(i)} \to \R$,
     \[ h^{(1)}(s) = -(n-1) \, , 
     \quad h^{(2)}(s) = -2(n-1) \frac{s}{\sqrt{1-s^2}} \, . \]
    Finally, putting everything together, the ODE~\eqref{eq:ode} in these cases reads, respectively, 
    \begin{align*}
        f''(s) =& \frac{\left(f'(s)^2+1\right) \left(n s f'(s)+1\right)}{s^2} \, , \\  \label{ode3}
        f''(s) =& \frac{\big(s^2-\left(s^2-1\right) f'(s)^2\big) \big(\left((2 n-3) s^2+1\right) f'(s)+s\big)+s^2 f'(s)}{s^3 \left(1-s^2\right)}.
    \end{align*}
Straightforward computations show that the ODE for the $\mathcal{G}^{(1)}$-solitons coincides with the one for the usual rotationally invariant translating solitons in $\R^{n+1}$ after the change of variable $t = 1/s$, and that we hence recover the rotationally invariant translating solitons studied in~\cite{MROSHS}. Another recent study of $\mathcal{G}^{(1)}$-solitons is~\cite{BL24}. 

To the knowledge of the authors, the only paper about $\mathcal{G}^{(2)}$-solitons is~\cite{LRS}, which deals with solitons in $\mathbb{H}^3$ using different methods. For the $\mathcal{G}^{(2)}$-solitons described here, we can employ the same technique as before: similarly to Lemma~\ref{ff2} and Lemma~\ref{extendtoinfinity}, we can construct a bowl solution and extend it by analysing an appropriate cubic polynomial and the critical points of the corresponding ODE solutions. It is also straightforward to establish an analogue of Lemma~\ref{wing-likesol}, yielding wing-like solitons. However, since this primarily involves simple ODE-solving and straightforward computations that follow the same method as before, we omit the details here, as they do not add significant insight. 
\end{remark}

\FloatBarrier

\appendix

\section{Some Technical Lemmas} 

In this appendix, we prove some technical results that allow us to obtain solitons in hyperbolic spaces in Section~\ref{sec:applications}. Let $I_0$, $I_1$ be two open intervals,  and let $p\in \mathcal{C}^1(I_0)$, $q\in \mathcal{C}^1(I_0\times I_1)$. We are interested in the solutions to the following ODE:
\begin{equation}\label{general-ode}
v'(s)=\frac{q(s,v(s))}{p(s)}.
\end{equation}
Define the vector field $X:I_0\times I_1\to\R^2$, $X(s,x)=(p(s),q(s,x))$.
The following result relates solutions to~\eqref{general-ode} and integral curves of $X$. In particular, we obtain a solution to~\eqref{general-ode} from almost any integral curve of $X$. 
\begin{lemma} \label{integralcurves}
\begin{enumerate}
\item Let $J_0$ be an open interval and $\beta:J_0 \to \R^2$, $\beta(r)=(s(r),x(r))$, an  integral curve of $X$ such that $s$ is a diffeomorphism onto its image. Then, $v:=x\circ s^{-1}$, is a solution to~\eqref{general-ode}. 
\item Given a solution $v$ to~\eqref{general-ode} defined on an open interval, there exists a smooth function $s=s(r)$ defined on some open interval such that $\beta(r)=(s(r),v(s(r)))$ is an integral curve of $X$.
\end{enumerate}
\end{lemma}
\begin{proof} Since $\beta$ is an integral curve of $X$, 
\begin{align*}s'(r)=p(s(r)), \quad
x'(r)=q(s(r),x(r)).
\end{align*}
Since $s$ is a diffeomorphism onto $J_0' := s(J_0)$, we have that $s'(s^{-1}(y))=p(y)$ for any $y\in J_0'$.
Therefore,
\begin{align*}
x'(s^{-1}(y)) & = q\big(s((s^{-1}(y)),x(s^{-1}(y)\big)=q(y,v(y)). \end{align*}
Next,
\begin{align*}
v'(y)=&\frac{x'(s^{-1}(y))}{s'(s^{-1}(y))} =
\frac{ q(y,v(y))}{p(y)}.
\end{align*}
This proves (1). For (2), let $s=s(r)$ be a solution to $s'(r)=p(s(r))$ in some open interval, and define the curve $\beta(r)=(s(r),v(s(r)))$. Since $v$ is a solution to~\eqref{general-ode}, we obtain
\begin{gather*}
\beta'(r)  = (s'(r), v'(s(r))s'(r)) = \left(p(s(r)),\frac{q(s(r),v(s(r)))}{p(s(r))}s'(r)\right)\\
=\big(p(s(r)),q(s(r),v(s(r)))\big)=X(\beta(r)).
\end{gather*}
This finishes the proof.\end{proof}

Now, we are going to solve some boundary problems related to~\eqref{general-ode}.
\begin{lemma}\label{boundarysol} Let $(s_0,x_0)\in I_0\times I_1$ such that $p(s_0)=0$, and $q(s_0,x_0)=0$. Assume that the following condition holds:
\begin{equation} \label{condition} p'(s_0)\frac{\partial q}{\partial x}(s_0,x_0)<0. \end{equation}
Then, the boundary problem
\begin{equation}\label{boundary-problem}
v'(s)=\frac{q(s,v(s))}{p(s)}, \quad v(s_0)=x_0,
\end{equation} has a unique solution $\vv:J\to\R$ defined on a suitable open interval $J \ni s_0$.
\end{lemma}
\begin{proof} We reuse the vector field $X:I_0\times I_1\to\R^2$, $X(s,x)=(p(s),q(s,x))$.
Note that $X(s_0,x_0)=(0,0)$. The differential of $X$ at $(s_0,x_0)$ is
\[ DX(s_0,x_0)=\begin{pmatrix} p'(s_0) & \frac{\partial q}{\partial s}(s_0,x_0)  \\ 0 &
\frac{\partial q}{\partial x}(s_0,x_0)
\end{pmatrix}.\]
Condition~\eqref{condition} implies that $DX(s_0,x_0)$ is diagonalisable, and its eigenvalues have different sign. By Section 3.2 of~\cite{W2003}, there are two $\mathcal{C}^{\infty}$ curves (the stable and the unstable submanifolds) passing through the point  $(s_0,x_0)$ and such that the tangent lines are spanned by the eigenvectors of $DX(s_0,x_0)^T$. One of them is parallel to $(0,1)$, so it cannot generate a solution as in Lemma~\ref{integralcurves}. However, the other curve can be parametrised by  $\beta(r)=(s(r),x(r))$ in some interval, with  $\beta(0)=(s_0,x_0)$, providing a solution to~\eqref{boundary-problem} as in Lemma~\ref{integralcurves}, namely  $\vv(y)=x(s^{-1}(y))$, with boundary condition $\vv(s_0)=x_0$. The uniqueness of (local) integral curves of $X$ implies that $\vv$ is unique.
\end{proof}

\section{Calculation of the mean curvature~\eqref{eq:mean_curvature} } \label{appendix:mean_curvature}

Recall from Section~\ref{sec:applications} the Riemannian submersion 
\[ 
\pi \colon \check{\mathbb{H}}^n \to (0, +\infty) \, , \quad \pi(x) = \frac{x_n}{x_1} \, , 
\] 
where $\check{\mathbb{H}}^n = \left\{ x \in \mathbb{H}^n \colon x_{n} > 0 \right\}$. In this appendix, we show that the mean curvature of the fibre $\pi^{-1 }(\{s\}) \subseteq \check{\mathbb{H}}^n$, $s>0$, with respect to the unit normal vector field $- \nabla \pi / \lVert \nabla \pi \rVert$ is given by
\[
         h(s) = (n-1) \frac{s}{\sqrt{1+s^2}} \, , 
\]
First, recall that the hyperbolic metric $g$ on $\check{\mathbb{H}}^n \subseteq \mathbb{H}^n$ is 
\[ 
g = e^{2 \gamma} \langle \cdot , \cdot\rangle \, , 
\]
where $\gamma = - \log{x_1}$ and $\langle \cdot , \cdot \rangle$ is the Euclidean metric. Hence, the corresponding Levi-Civita connections $\nabla$ and $\nabla^{\mathrm{Eucl}}$ are related by the usual conformal change formula 
\begin{align} \label{eq:LC_conformal}
\nabla_X Y &= \nabla^{\mathrm{Eucl}}_X Y + X(\gamma) Y + Y(\gamma) X - \langle X,Y \rangle \grad_{\mathrm{Eucl}} (\gamma) \notag \\
&= \nabla^{\mathrm{Eucl}}_X Y + X(\gamma) Y + Y(\gamma) X + \frac{1}{x_1} \langle X,Y \rangle \partial_{1} \, .
\end{align}
Now, let $e_i = x_1 \partial_i$, so that $(e_1 , \dots , e_n)$ is a global $g$-orthonormal frame of $\mathbb{H}^n$. The gradient of $\pi$ is given by  
\begin{align*}
\nabla \pi &= \sum_{i=1}^{n} g\left( \nabla \pi , e_i \right) e_i = x_1^2\sum_{i=1}^{n} g\left( \nabla \pi , \partial_i \right) \partial_i = x_1^2\sum_{i=1}^{n} \frac{\partial \pi}{\partial x_i} \partial_i \\
&= x_1^2 \left( -\frac{x_n}{x_1^2} \partial_1 + \frac{1}{x_1} \partial_n \right) = - x_n \partial_1 + x_1 \partial_n \, . 
\end{align*}
Thus, the square of its norm is 
\[
\lVert \nabla \pi \rVert^2 = g \left( \nabla \pi , \nabla \pi \right) = 1 + \left( \frac{x_n}{x_1} \right)^2 = 1+s^2 \, , 
\]
which proves equation~\eqref{eq:alpha}. A unit normal vector field to the fibres of $\pi$ is then 
\[ 
\xi := -\frac{\nabla \pi}{\lVert \nabla \pi \rVert} = \frac{x_1^2}{\sqrt{x_1^2 + x_n^2}} \left( s \partial_1 - \partial_n \right) \, . 
\]
An orthonormal frame of the fibre $\pi^{-1}(\{s\})$ is given by 
\[ 
\left( x_1 \partial_2 , \dots , x_1 \partial_{n-1} , T := \frac{x_1}{\sqrt{1+s^2}} \left( \partial_1 + s \partial_n \right) \right) \, . 
\]
Hence, the mean curvature of the fibre with respect to $\xi$ is given by 
\begin{align*}
h &= \sum_{j=2}^{n-1} g \left( \nabla_{x_1 \partial_j} x_1 \partial_j , \xi  \right) + g\left( \nabla_{T} T , \xi \right) = x_1^2 \sum_{j=2}^{n-1} g \left( \nabla_{\partial_j} \partial_j , \xi  \right) + g\left( \nabla_{T} T , \xi \right) \\
&= x_1^2 \sum_{j=2}^{n-1} g \left( \frac{1}{x_1} \partial_1 , \xi  \right) + g\left( \nabla_{T} T , \xi \right) = (n-2) \frac{s}{\sqrt{1+s^2}} +  g\left( \nabla_{T} T , \xi \right) \, , 
\end{align*}
where the third equality uses~\eqref{eq:LC_conformal}. Finally, a straightforward calculation using~\eqref{eq:LC_conformal} again shows that  
\begin{align*}
\nabla_T T = \frac{x_1 x_n}{x_1^2 + x_n^2} \left( x_n \partial_1 -  x_1 \partial_n \right) \, . 
\end{align*}
Putting everything together, we conclude the proof of equation~\eqref{eq:mean_curvature}. 

\section*{Acknowledgements}

D.~Artacho is funded by the UK Engineering and Physical Sciences Research Council (EPSRC), grant EP/W523872.

M.~Ortega is partially financed by: (1) the Spanish MICINN and ERDF project PID2020-116126GB-I00; and (2) the “Maria de Maeztu” Excellence Unit IMAG, ref. CEX2020-001105-M, funded by MCIN/AEI/10.13039/501100011033.

\bibliographystyle{alphaurl}
\bibliography{references.bib}

\begin{thebibliography}{MdOSHdS24}

\bibitem[ALO22]{ALO}
D.~Artacho, M.-A. Lawn, and M.~Ortega.
\newblock {Translating Solitons in Generalised Robertson-Walker Geometries}.
\newblock 2022.
\newblock URL: \url{https://arxiv.org/pdf/2211.14529}.

\bibitem[ALR19]{ALR}
L.~Alias, J.H. Lira, and M.~Rigoli.
\newblock Mean curvature flow solitons in the presence of conformal vector fields.
\newblock {\em M. J. Geom. Anal.}, pages 1--64, 2019.
\newblock \href {https://doi.org/10.1007/s12220-019-00186-3} {\path{doi:10.1007/s12220-019-00186-3}}.

\bibitem[BL24]{BL24}
A.~Bueno and R.~L{\'o}pez.
\newblock A new family of translating solitons in hyperbolic space, 2024.
\newblock URL: \url{https://arxiv.org/abs/2402.05533}, \href {https://arxiv.org/abs/2402.05533} {\path{arXiv:2402.05533}}.

\bibitem[BL25]{BL25}
A.~Bueno and R.~L\'opez.
\newblock The class of grim reapers in {$\mathbb{H}^2\times \mathbb{R}$}.
\newblock {\em J. Math. Anal. Appl.}, 541(2):Paper No. 128730, 23, 2025.
\newblock \href {https://doi.org/10.1016/j.jmaa.2024.128730} {\path{doi:10.1016/j.jmaa.2024.128730}}.

\bibitem[CMR20]{CMR}
G.~Colombo, L.~Mari, and M.~Rigoli.
\newblock Remarks on mean curvature flow solitons in warped products.
\newblock {\em Discrete Contin. Dyn. Syst. Ser.}, 13(7):1957–1991, 2020.
\newblock \href {https://doi.org/10.3934/dcdss.2020153} {\path{doi:10.3934/dcdss.2020153}}.

\bibitem[CSS07]{CSS}
J.~Clutterbuck, O.~C. Schn\"urer, and F.~Schulze.
\newblock Stability of translating solutions to mean curvature flow.
\newblock {\em Calc. Var.}, 29(3):281--293, 2007.
\newblock \href {https://doi.org/10.1007/s00526-006-0033-1} {\path{doi:10.1007/s00526-006-0033-1}}.

\bibitem[dLRdS24]{LRS}
R.~F. de~Lima, A.~K. Ramos, and J.~P. dos Santos.
\newblock Solitons to mean curvature flow in the hyperbolic 3-space, 2024.
\newblock URL: \url{https://arxiv.org/abs/2307.14136}, \href {https://arxiv.org/abs/2307.14136} {\path{arXiv:2307.14136}}.

\bibitem[DS23]{DS}
I.~Domingos and M.~Santos.
\newblock Aspects of mean curvature flow solitons in warped products.
\newblock {\em Results Math}, 78(191), 2023.
\newblock \href {https://doi.org/10.1007/s00025-023-01969-5} {\path{doi:10.1007/s00025-023-01969-5}}.

\bibitem[EH89]{H}
K.~Ecker and G.~Huisken.
\newblock Mean curvature evolution of entire graphs.
\newblock {\em Ann. of Math. (2)}, 130(3):453--471, 1989.
\newblock \href {https://doi.org/10.2307/1971452} {\path{doi:10.2307/1971452}}.

\bibitem[HIMW19]{HIMW}
D.~Hoffman, T.~Ilmanen, F.~Martín, and B.~White.
\newblock Graphical translators for mean curvature flow.
\newblock {\em Calc. Var.}, 58(117), 2019.
\newblock \href {https://doi.org/10.1007/s00526-019-1560-x} {\path{doi:10.1007/s00526-019-1560-x}}.

\bibitem[HS00]{HS00}
N.~Hungerb\"uhler and K.~Smoczyk.
\newblock Soliton solutions for the mean curvature flow.
\newblock {\em Differential Integral Equations}, 13(10-12):1321--1345, 2000.
\newblock \href {https://doi.org/10.57262/die/1356061128} {\path{doi:10.57262/die/1356061128}}.

\bibitem[LM19]{LM}
J.~H. Lira and F.~Martin.
\newblock Translating solitons in riemannian products.
\newblock {\em J. Diff. Equations}, 266(22):7780--7812, 2019.
\newblock \href {https://doi.org/10.1016/j.jde.2018.12.015} {\path{doi:10.1016/j.jde.2018.12.015}}.

\bibitem[LO22]{LO}
M-A. Lawn and M.~Ortega.
\newblock Translating solitons in a lorentzian setting, submersions and cohomogeneity one actions.
\newblock {\em Mediterranean Journal of Mathematics}, 19(3):1025–1053, 2022.
\newblock \href {https://doi.org/10.1007/s00009-022-02020-7} {\path{doi:10.1007/s00009-022-02020-7}}.

\bibitem[LTW11]{LTW}
G.~Li, D.~Tian, and C.~Wu.
\newblock Translating solitons of mean curvature flow of noncompact submanifolds.
\newblock {\em Differential Integral Equations}, 14(83), 2011.
\newblock \href {https://doi.org/10.1007/s11040-011-9088-0} {\path{doi:10.1007/s11040-011-9088-0}}.

\bibitem[MdOSHdS24]{MROSHS}
L.~Mari, J.~Rocha de~Oliveira, A.~Savas-Halilaj, and R.~Sodr\'e de~Sena.
\newblock Conformal solitons for the mean curvature flow in hyperbolic space.
\newblock {\em Ann. Glob. Anal. Geom.}, 65(19), 2024.
\newblock \href {https://doi.org/10.1007/s10455-024-09947-y} {\path{doi:10.1007/s10455-024-09947-y}}.

\bibitem[MSHS15]{MSHS}
F.~Mart\'in, A.~Savas-Halilaj, and K.~Smoczyk.
\newblock On the topology of translating solitons of the mean curvature flow.
\newblock {\em Calc. Var.}, 54, 2015.
\newblock \href {https://doi.org/Calc. Var. 54, 2853–2882 (2015). https://doi.org/10.1007/s00526-015-0886-2} {\path{doi:Calc. Var. 54, 2853–2882 (2015). https://doi.org/10.1007/s00526-015-0886-2}}.

\bibitem[Wig03]{W2003}
S.~Wiggins.
\newblock {\em Introduction to Applied Nonlinear Dynamical Systems and Chaos, second edition}, volume~2 of {\em Texts in Applied Mathematics}.
\newblock Springer-Verlag, New York, 2003.
\newblock \href {https://doi.org/10.1007/b97481} {\path{doi:10.1007/b97481}}.

\end{thebibliography}

\end{document}